\documentclass[a4paper, reqno, 12pt]{amsart}

\usepackage[utf8]{inputenc}
\usepackage{amsthm,amsfonts,amssymb,amsmath,amsxtra,amsrefs}
\usepackage{graphicx}
\usepackage{mathtools}
 \usepackage{bbold}
\usepackage{latexsym}
 \usepackage{stmaryrd}
\usepackage{graphicx}
\usepackage{tikz}
\usepackage{tikz-cd}
\usepackage{todonotes,cancel}
\usepackage[margin=0.8in]{geometry}
\usepackage{mathrsfs}
\usepackage{bm}

\usepackage[all]{xy}
\SelectTips{cm}{}
\usepackage{xr-hyper}
\usepackage[colorlinks=false,
   citecolor=Black,
   linkcolor=Red,
   urlcolor=Blue]{hyperref}
\usepackage{verbatim}
\RequirePackage{xspace}
\RequirePackage{etoolbox}
\RequirePackage{varwidth}
\RequirePackage{enumitem}
\RequirePackage{tensor}
\RequirePackage{mathtools}
\RequirePackage{longtable}
\RequirePackage{multirow}

\usepackage{rotating}

\tikzset{
    labl/.style={anchor=south, rotate=90, inner sep=.5mm}
}

\newtheorem{thm}{Theorem}

\newtheorem{thmintro}{Theorem}

\newtheorem{prop}[thm]{Proposition}

\newtheorem{lem}[thm]{Lemma}

\newtheorem{cor}[thm]{Corollary}

\theoremstyle{definition}

\newtheorem{rem}[thm]{Remark}
\newtheorem{remark}[thm]{Remark}

\numberwithin{equation}{section}
\numberwithin{thm}{section}

\newcommand{\I}{I}
\newcommand{\coroot}{\alpha^\vee}
\newcommand{\A}{\mathcal{A}}
\newcommand{\U}{\mathrm{U}}

\newcommand{\AU}{{_\mathcal{A}\dot{\U}}}
\newcommand{\Ui}{\U^\imath}
\newcommand{\Udot}{\dot{\U}}
\newcommand{\Uidot}{\dot{\U}^\imath}

\def\lr#1#2{\ensuremath{\left(\kern-.3em\left(\genfrac{}{}{0pt}{}{#1}{#2}\right)\kern-.3em\right)}}

\newcommand{\Qq}{\mathbb{Q}(q)}

\newcommand{\agUi}{\U^{\imath,+}}

\newcommand{{\BG}}{\ensuremath{\mathbb {G}}\xspace}

\newcommand{{\BK}}{\ensuremath{\mathbb {K}}\xspace}

\newcommand{\BQ}{\ensuremath{\mathbb {Q}}\xspace}

\newcommand{\BZ}{\ensuremath{\mathbb {Z}}\xspace}

\newcommand{\CA}{\ensuremath{\mathcal {A}}\xspace}

\newcommand{\RB}{\ensuremath{\mathrm {B}}\xspace}

\begin{document}

\title[]{Coordinate rings on symmetric spaces}

\subjclass[2020]{} 

\begin{abstract}
Let $G_k$ be a connected reductive group over an algebraically closed field $k$ of char $\neq 2$. Let $\theta_k$ be an algebraic group involution of $G_k$ and denote the fixed point subgroup by $K_k$. 
We construct an integral model for the symmetric space $K_k \backslash G_k$ with a natural action of the Chevalley group scheme over integers. We show the coordinate ring $k[K_k \backslash G_k]$ admits a canonical basis, as well as a good filtration as a $G_k$-module. We also construct a canonical basis and an integral form for the space of $K_k$-biinvariant functions on $k[G_k]$. Our results rely on the construction of quantized coordinate algebras of symmetric spaces, using the theory of canonical bases on quantum symmetric pairs.
\end{abstract}

\author[Huanchen Bao]{Huanchen Bao}
\address{Department of Mathematics, National University of Singapore, Singapore.}
\email{huanchen@nus.edu.sg}

\author[Jinfeng Song]{Jinfeng Song}
\address{Department of Mathematics, National University of Singapore, Singapore.}
\email{j\_song@u.nus.edu}

	\maketitle
\section{Introduction}

\subsection{}
     Let $G_k$ be a connected reductive group over an algebraically closed field $k$. We denote the coordinate ring of $G_k$ by $k[G_k]$. It is known \cite{Lu07} that  $k[G_k]$ admits an integral form ${}_\BZ\mathbf{O}$ that defines the Chevalley group scheme $G_\BZ$ over $\BZ$ parameterizing split reductive groups of the given type. Lusztig \cite{Lu07} gave a construction of  the ring ${}_\BZ \mathbf{O}$ using his theory of canonical bases on quantum groups.

     Assume char $k \neq 2$ and let $\theta_k$ be an algebraic group involution of the reductive group $G_k$. The classification of such involutions is independent of the characteristic of $k$ (provided $\neq 2$) by Springer \cite{Spr87}. 
     We denote the fixed point subgroup by $K_k$. In our previous work \cite{BS22}, we have constructed the symmetric subgroup scheme over $\BZ$ parameterizing such symmetric subgroups.
     
       We study the symmetric space $K_k \backslash G_k$ in this paper. Let $k[K_k \backslash G_k]$ be the coordinate ring. The first main result of this paper is the construction of an integral form of $k[K_k \backslash G_k]$, hence an integral model parameterizing the symmetric spaces $K_k \backslash G_k$ for various algebraically closed field $k$.

       \begin{thmintro}[Theorem~\ref{thm:base} \& Theorem~\ref{thm:KG}]\label{thm:1}
There exists a commutative ring ${}_\BZ \mathbf{O}(K \backslash G) \subset {}_\BZ \mathbf{O}$ over $\BZ$, such that ${}_k \mathbf{O}(K \backslash G) = k \otimes_\BZ {}_\BZ \mathbf{O}(K \backslash G) \cong k[K_k \backslash G_k]$ for any algebraically closed field of char $\neq 2$.    

There exists a natural basis $B(K \backslash G)$ on ${}_\BZ \mathbf{O}(K \backslash G)$ which we call the canonical basis.
       \end{thmintro}
We actually construct a quantization ${}_\CA \mathbf{O}(K \backslash G)$ of the coordinate ring ${}_\BZ \mathbf{O}(K \backslash G)$ over the ring $\CA= \BZ[q,q^{-1}]$, where ${}_\BZ \mathbf{O}(K \backslash G)$ is obtained via a specialization of $q$ at $1$.

The desire to construct quantum homogeneous spaces was the original motivation to study quantum symmetric pairs \cites{Noumi, NS}. However, the existing construction of quantum homogeneous spaces was entirely over the rational field $\Qq$; cf. \cite{Le03}. So it can not be specialized even to recover the $q=1$ classical case. In order to construct the $\CA$-form, one is forced to utilize the theory of canonical bases on quantum symmetric pairs developed by the first author and Wang in \cite{BW18a}.

     \subsection{} 
 It is a classical result that the coordinate ring $k[G_k]$ admits a direct sum decomposition $k[G_k] \cong \bigoplus_{\lambda \in X^+} {}_k V(\lambda)^* \otimes {}_k V(- w_0\lambda)^*$ as $G_k \times G_k$-modules if $k$ is of characteristic $0$. 
 In positive characteristics, the coordinate ring $k[G_k]$ is not semisimple anymore. In this case, we instead have a filtration of $k[G_k]$  by $G_k \times G_k$-submodules indexed by dominate weights in $X^+$, such that $k[G_k]_{\le \lambda} / k[G_k]_{< \lambda} \cong {}_k V(\lambda)^* \otimes {}_k V(- w_0\lambda)^*$ for any $\lambda \in X^+$. This is often referred as a good filtration \cite{Jan03}*{\S4.20} of $k[G_k]$.

    In light of Lusztig's construction of the coordinate ring $k[G_k]$ using quantum groups, the good filtration on $k[G_k]$ can then be obtained via dualizing the cell filtration \cite{Lu94}*{Chapter~29} on quantum groups. 

Assume again char $k \neq 2$. The coordinate ring $k[K_k \backslash G_k]$ of the symmetric space $K_k \backslash G_k$ is naturally a $G_k$-module. If $k$ is of characteristic $0$, it follows from the direct sum decomposition of $k[G_k]$ that $k[K_k \backslash G_k] \cong \bigoplus_{\lambda \in X_\imath^+}  {}_k V(- w_0\lambda)^*$ as a $G_k$-module.  Here $X_\imath^+$ is a subset of $X^+$ consisting of spherical dominant weights (See \S\ref{sec:spherical}).  
Little is known for positive characteristics.

    We state the second main result of this paper. 

    \begin{thmintro}[Theorem~\ref{thm:fil}]\label{thm:2}    The $G_k$-module $k[K_k \backslash G_k]$ admits a good filtration indexed by $ X_\imath^+$, such that 
    \[
         k[K_k \backslash G_k]_{\le \lambda} / k[K_k \backslash G_k]_{< \lambda} \cong {}_k V(- w_0\lambda)^*, \quad \text{for any }\lambda \in X_\imath^+.
    \]
    \end{thmintro}

\subsection{} The construction of both the commutative ring ${}_\BZ \mathbf{O}(K \backslash G)$ (or the quantization ${}_\CA \mathbf{O}(K \backslash G)$) and the good filtration of $k[K_k \backslash G_k]$ relys on the study of  based modules of $\imath$quantum groups. 

Let $(\U, \Ui)$ be a quantum symmetric pair \cite{BW18a} associated with the symmetric pair $(G_k, K_k)$. The subalgebra $\Ui \subset \U$ of the quantum group $\U$ is refered as the $\imath$quantum group.

We now state the third main result of this paper on based $\Ui$-modules, which is the technical foundation for both Theorem \ref{thm:1} and Theorem \ref{thm:2}.

\begin{thmintro}[Theorem~\ref{thm:Fbased}]
    Let $M$ be a finite-dimensional based $\U$-module. Let $\agUi$ be the augmentation ideal of $\Ui$. Then $M_+=\agUi M$ is a based $\Ui$-submodule of $M$, and the coinvariant space $M/ M_+$ is naturally a based $\Ui$-module.
\end{thmintro}

Based modules and based morphisms are well-behaved with respect to specilizations. 
We are then able to deduce classical results on symmetric spaces via specializations.

\subsection{} In a subsequent paper, we shall apply results in this paper to construct the Vinberg enveloping variety of the symmetric space $K_k \backslash G_k$ and an integral model of the wonderful compactification of $K_k \backslash G_k$. 

    \vspace{.2cm}
\noindent {\bf Acknowledgement: } Both authors are supported by MOE grants A-0004586-00-00 and A-0004586-01-00.

\section{Preliminaries}
\subsection{Quantum groups}\label{sec:qp}
		We recall quantum groups following Lusztig's book \cite{Lu94}.
	
	\subsubsection{Definitions}\label{sec:def} 
	Let $(\I,Y,X,A,(\alpha_i)_{i\in\I},(\coroot_i)_{i\in\I})$ be a root datum of finite type; cf. \cite{BS22}*{\S2.1}. Let $D=diag(\varepsilon_i)_{i\in\I}$ be the diagonal matrix such that $DA$ is a symmetric matrix, with $\epsilon_i\in\mathbb{\BZ}_{>0}$   and $\{\varepsilon_i\mid i\in\I\}$ are relatively prime. Let $W=\langle s_i\mid i\in I\rangle$ be the Weyl group associated with the root datum with the longest element denoted by $w_0$. For $\lambda,\mu\in X$, we write $\lambda\leq \mu$ if $\mu-\lambda\in\sum_{i\in I}\mathbb{Z}_{\geq 0}\alpha_i$. Let $X^+=\{\mu\in X\mid\langle\alpha_i^\vee,\mu\rangle\geq 0,\forall i\in\I\}$ be the set of \emph{dominant weights}.
 
 Let $q$ be an indeterminate. We write $\A=\mathbb{Z}[q,q^{-1}]$ to be the subring of $\mathbb{Q}(q)$. Let $\U$ be the Drinfeld--Jimbo quantum group associated with the root datum; cf. \cite{Lu94}*{\S 3.1}. Recall that $\U$ is an associated $\Qq$-algebra with 1. It has generators $
	E_i$ $(i\in\I)$, $F_i$ $(i\in\I)$ and $ K_\mu$  $( \mu\in Y)$. The algebra $\U$ has a structure of a Hopf algebra. Let $\varepsilon:\U\rightarrow\Qq$ be the counit, and $\Delta:\U\rightarrow \U\otimes \U$ be the coproduct; cf. \cite{Lu94}*{Lemma 3.1.4 \& \S3.1.11}.
 
Following \cite{Lu94}*{\S 3.1.3}, let $\sigma:\U\rightarrow\U$ be the algebra anti-automorphism such that $\sigma(E_i)=E_i$, $\sigma(F_i)=F_i$, for $i\in I$, and $\sigma(K_\mu)=K_{-\mu}$, for $\mu\in Y$, and let $\omega:\U\rightarrow\U$ be the algebra automorphism such that $\omega(E_i)=F_i$, $\omega(F_i)=E_i$, for $i\in I$, and $\omega(K_\mu)=K_{-\mu}$, for $\mu\in Y$.

For any $\lambda \in X^+$, we denote by $V(\lambda)$ the highest weight simple $\U$-module with highest weight $\lambda$. We denote the highest weight vector by $v^+_\lambda$ and the canonical basis of $V(\lambda)$ by $\RB(\lambda)$.

\subsubsection{Based modules}\label{sec:based}
Let $\mathbf{A} = \Qq \cap \BQ[[q^{-1}]]$ be the ring of rational functions which are regular at $q=\infty$. Let $(M,B)$ be a finite-dimensional based $\U$-module over $\Qq$; cf. \cite{Lu94}*{\S27}. We denote by $_\A M$ (resp., $L(M)$) the $\A$-submodule (resp., $\mathbf{A}$-submodule) of $M$ spanned by $B$. For any commutative $\CA$-algebra $R$, we write ${}_R M = R \otimes_{\A} {}_\A M$ and still denote the basis by $B$. We often consider a commutative ring $R$ as an $\A$-algebra via the ring homomorphism $\A \rightarrow R$, $q \mapsto 1$ unless specified otherwise.

Following \cite{Lu94}*{Proposition~27.1.8}, for any $\lambda \in X^+$, we have the based submodule $M[\ge \lambda] \subset M$ (resp., $M[> \lambda] \subset M$) with basis $B[\ge \lambda] = B \cap M[\ge \lambda]$ (resp., $B[> \lambda] = B \cap M[> \lambda]$). We set $B[\lambda] = B[\ge \lambda] - B[> \lambda]$. Let $B[\lambda]^{hi}$ be the subset $B[\lambda]$ defined in \cite{Lu94}*{\S27.2.3}.

For $\lambda\in X^+$, let $\pi_\lambda: M \rightarrow M / M[\ge \lambda]$ be the canonical projection map. Then $M / M[\ge \lambda]$ is naturally a based $\U$-module with a basis consisting of the image of $B - B[\ge \lambda]$ under the map $\pi_\lambda$; see \cite{Lu94}*{\S27.1.4}. In particular, the map $\pi_\lambda: M \rightarrow M / M[\ge \lambda]$ is based in the sense of \cite{Lu94}*{\S27.1.3}. We often abuse notations and denote by $B - B[\ge \lambda]$ the basis of $M / M[\ge \lambda]$.

For any $\lambda \in X^+$, the subquotient $M[\ge \lambda] / M[> \lambda]$ is isomorphic to the based module $(\bigoplus_{i=1}^{n_\lambda} V(\lambda), B')$, where $B'$ is given by the union of the canonical bases of the various summands $V(\lambda)$ and $n_\lambda$ is the cardinality of the subset $B[\lambda]^{hi}\subset B$.

We conclude that 

{\it (a) the based module $(M,B)$ admits a based filtration 
\begin{equation*} 
0 = M[\ge \lambda_0] \subset M[\ge \lambda_1] \subset M[\ge \lambda_2] \subset \cdots \subset M[\ge \lambda_n] = M, \quad \text{for }\lambda_i \in X^+,
\end{equation*}
such that, for $1 \le i \le n$, the map $\phi_{\lambda_i}:M[\ge \lambda_i]/M[\ge \lambda_{i-1}] \overset{\sim}{\rightarrow}\bigoplus_{b \in B[\lambda_i]^{hi}} V(\lambda_i)$, $\Bar{b}\mapsto (v_{\lambda_i}^+)_b$, for $b\in B[\lambda_i]^{hi}$, where $(v_{\lambda_i}^+)_b$ is the highest vector of the copy corresponding to $b$ of $V(\lambda_i)$, defines an isomorphism of based $\U$-modules. 
}

The following lemma follows from \cite{Ka91}*{lemma~2.6.3}. 
\begin{lem}\label{lem:Ka91}
Let $(M,B)$ be a finite-dimensional based $\U$-module. We have an isomorphism of $\U$-modules 
\[
\phi: M \longrightarrow \bigoplus_{i =1}^n \bigoplus_{b \in B[\lambda_i]^{hi}} V(\lambda_i),
\]
such that (for any $1 \le j\le n$),
\begin{enumerate}
\item $\phi$ restricts to an isomorphism $M[\ge \lambda_j] \longrightarrow \bigoplus_{i =1}^j \bigoplus_{b \in B[\lambda_i]^{hi}} V(\lambda_i)$;
\item the induced morphism on $M[\ge \lambda_j]/M[\ge \lambda_{j-1}]$ agrees with $\phi_{\lambda_j}$ in \S\ref{sec:based} (a);
\item for any $b \in B[\lambda_j]$, we have $\phi(b) \in \phi_{\lambda_j}(b) + q^{-1} \phi(L(M[\ge \lambda_{j-1}]))$;
\item $\phi$ restricts to isomorphisms as $\mathbf{A}$-modules $L(M[\ge \lambda_j]) \longrightarrow \bigoplus_{i =1}^j \bigoplus_{b \in B[\lambda_i]^{hi}} L(V(\lambda_i))$.
\end{enumerate}
\end{lem}

\subsubsection{The modified form}	\label{sec:modiU}
	 Let $\dot{\U}$ be the modified quantum group \cite{Lu94}*{\S 23.1} and $\dot{\RB}$ be its canonical basis (\cite{Lu94}*{\S 25.2.1}). The algebra (without unit) $\dot{\U}$ is a $(\U,\U)$-bimodule in a natural way. Let $\AU$ be the free $\A$-submodule of $\dot{\U}$ with basis $\dot{\RB}$. 
        A unital $\Udot$-module is equivalent to a $\U$-module with a direct sum decomposition into weight spaces by \cite{Lu94}*{\S23.1.4}. These are the only modules we shall consider in this paper.

Lusztig \cite{Lu94}*{\S29.1} defined a partition of the canonical basis $
     \dot{\RB} = \sqcup_{\lambda \in X^+} \dot{\RB}[\lambda]$.
For $\lambda\in X^+$, let $\dot{\U}[\not \le\lambda]$ (resp., $_\A\dot{\U}[\not\le\lambda]$) be the $\Qq$-subspace (resp., $\A$-submodule) of $\dot{\U}$ spanned by $\sqcup_{\lambda'\not\le\lambda}\dot{\RB}[\lambda']$. Then $\Udot[\not\le\lambda]$ (resp., ${}_\mathcal{A}\Udot[\not\le\lambda]$) is a two-sided ideal of $\Udot$ (resp., ${}_\mathcal{A}\Udot$). 

The quotient $\Udot/\Udot[\not\le\lambda]$ is a based $\U$-module with the canonical basis consisting of the image of $\dot{\RB}[\le\lambda]=\sqcup_{\lambda'\le\lambda}\dot{\RB}[\lambda']$. We abuse notations, and denote the quotient based module by $(\Udot/\Udot[\not\le\lambda], \dot{\RB}[\le\lambda])$. Similarly we define the subspace $\dot{\U}[<\lambda]$, the $\A$-submodule $_\A\dot{\U}[<\lambda]$, and the quotient based module $(\dot{\U}/\dot{\U}[\not <\lambda],\dot{\RB}[<\lambda])$.

\subsubsection{Coordinate rings}
Let $\mathbf{O}$ be the $\Qq$-vector space of all $\Qq$-linear forms $f: \Udot \rightarrow \Qq$ with the following property: $f$ vanishes on $\Udot[\not\leq\lambda ]$ for some $\lambda \in X^+$. For any $a \in \dot{\RB}$, we define the linear form $\tilde{a}: \Udot \rightarrow \Qq$, given by $\tilde{a}(a') = \delta_{a, a'}$, for $a'\in\dot{\RB}$. The set $\{\tilde{a} \vert a \in \dot{\RB}\}$ is a $\Qq$-basis of $\mathbf{O}$, called the \emph{canonical basis} of $\mathbf{O}$. Note that elements in $\mathbf{O}$ can be naturally viewed as linear forms of the completion space $\widehat{\U}$, where $\widehat{\U}$ is the $\Qq$-vector space consisting of formal sums of $\Qq$-linear combinations of $\dot{\RB}$.

Let ${}_\CA \mathbf{O}$ be the $\CA$-submodule of $\mathbf{O}$, spanned by the canonical basis. Then by \cite{Lu94}*{\S29.5.2}, ${}_\CA \mathbf{O}$ is a Hopf algebra over $\CA$. For any commutative $\CA$-algebra $R$ with $1$, we write ${}_R \mathbf{O} = R \otimes_{\CA} {}_\CA \mathbf{O}$. Unless specified otherwise, we often consider a commutative ring $R$ as an $\CA$-algebra via the ring homomorphism $\CA \rightarrow R$, $q \mapsto 1$.

\begin{thm}\cite{Lu07}\label{thm:chs}
 The group scheme $Spec\, {}_\BZ \mathbf{O}$ is isomorphic to the Chevalley group scheme $\mathbf{G}_\BZ$ over $\BZ$ associated with the given root datum. 
 
 In particular, the ring ${}_k \mathbf{O}$ is canonically isomorphic to the coordinate ring $k[G_k]$ of the connected reductive algebraic group $G_k$ associated with the given root datum, for any algebraically closed field $k$.
\end{thm}

We next recall the Peter--Weyl filtration of $\mathbf{O}$ and ${}_\CA  \mathbf{O}$. For $\lambda \in X^+$, we define 
\[
      \mathbf{O}_{\le \lambda}   = {\rm Hom}_{\Qq}(\Udot/ \Udot[ \not \le \lambda], \Qq), \quad {}_\CA \mathbf{O}_{\le \lambda} = \mathbf{O}_{\le \lambda} \cap  {}_\CA \mathbf{O}=\text{Hom}_\A\big({}_\A(\dot{\U}/\dot{\U}[\not\le\lambda]),\A\big).
\]
Then $\mathbf{O}_{\le \lambda}$ (resp., ${}_\CA \mathbf{O}_{\le \lambda}$) is the $\Qq$-span (resp., $\CA$-span) of $\{\tilde{a} \vert a \in \dot{\RB}[\le \lambda]\}$. We define $\mathbf{O}_{<\lambda}$ and $_\A\mathbf{O}_{<\lambda}$ in the similar way.  

We then have 
\[
 \mathbf{O} = \cup_{\lambda \in X^+} \mathbf{O}_{\le \lambda}, \quad {}_\CA \mathbf{O} = \cup_{\lambda \in X^+} {}_\CA \mathbf{O}_{\le \lambda}.
\]
For an algebraically closed field $k$ and $\lambda\in X^+$, set $_k\mathbf{O}_{\leq\lambda}=k\otimes_\A {}_\A\mathbf{O}_{\leq\lambda}$. For any $\lambda \le \lambda'$, since the inclusion ${}_\CA \mathbf{O}_{\le \lambda} \rightarrow {}_\CA \mathbf{O}_{\le \lambda'}$  is compatible with bases, we have an inclusion ${}_k \mathbf{O}_{\le \lambda} \rightarrow {}_k \mathbf{O}_{\le \lambda'}$ as well. Therefore we have ${}_k \mathbf{O}_{\le \lambda} \subset {}_k\mathbf{O}$ for any $\lambda \in X^+$ and $_k\mathbf{O}=\cup_{\lambda\in X^+}{}_k\mathbf{O}_{\le\lambda}$.

Let $G_k$ be the connected reductive group as in Theorem \ref{thm:chs}. We define $G_k \times G_k$-module structure on $_k\mathbf{O} \cong k [G_k]$ as follows 
\[
 ((g_1, g_2) \cdot f )(x) = f(g_1^{-1} x g_2), \text{ for } (g_1, g_2) \in G_k \times G_k, f \in k [G_k], x \in G_k.
\]
Let ${}_kV(\lambda)^*$ be the dual Weyl module for any $\lambda \in X^+$. 
Dualizing the construction in \cite{Lu94}*{\S 29.3}, we obtain the following classical Peter--Weyl filtration; cf. \cite{Jan03}*{Proposition~4.20}. 

\begin{thm}\label{thm:PW}
For any $\lambda \in X^+$, the subspaces  $_k\mathbf{O}_{\le \lambda}$ and $_k\mathbf{O}_{< \lambda}$ are $G_k \times G_k$-submodules of $_k\mathbf{O}$ such that 
\[
     _k\mathbf{O}_{\le \lambda}/ _k\mathbf{O}_{< \lambda} \cong {}_kV(\lambda)^* \otimes_k {}_kV(-w_0\lambda)^*.
\]
In particular, the $G_k \times G_k$-module $_k\mathbf{O}$ admits a good filtration. 
\end{thm}

\subsection{Quantum symmetric pairs}

\subsubsection{$\imath$Quantum groups}\label{sec:iqp}

Let $(\I,Y,X,A,(\alpha_i)_{i\in\I},(\coroot_i)_{i\in\I}), I_\bullet, \tau)$ be an $\imath$root datum; cf. \cite{BS22}*{\S2.3.1}. Let $W_{I_\bullet}=\langle s_i\mid i\in I_\bullet\rangle$ be the parabolic subgroup associated with $I_\bullet$, and let $w_\bullet$ be the longest element of $W_{I_\bullet}$. Denote by  $\theta = -w_\bullet \circ \tau$ the involution on both $X$ and $Y$, by slightly abuse of notation. Set 
	\begin{equation*}
	X_\imath=X/\langle \lambda-\theta\lambda\mid\lambda\in X\rangle\qquad \text{and}\qquad Y^\imath=\{\mu\in Y\mid \theta\mu=\mu\}.
	\end{equation*}
	We call $X_\imath$ the \emph{$\imath$weight lattice}, and call $Y^\imath$ the \emph{$\imath$coweight lattice}. There is a natural pairing $Y^\imath\times X_\imath\rightarrow\mathbb{Z}$ inherited from the pairing between $Y$ and $X$. Note that this pairing is not necessarily perfect. For any $\lambda\in X$, we write $\overline{\lambda}$ to denote its image in $X_\imath$.

	For each $i\in I_\circ = I - I_\bullet$, we take a parameter $\varsigma_i\in \pm q^{\mathbb{Z}}$. Let $\U^\imath=\U^\imath_{\mathbf{\varsigma}} \subset \U$ be the $\imath$quantum group defined in \cite{BW18a}*{\S3.3}. By definition, $\U^\imath$ is the $\mathbb{Q}(q)$-subalgebra of $\U$ generated by elements
	\begin{equation*}
	B_i=F_i+\varsigma_i T_{w_{\bullet}}(E_{\tau i})\Tilde{K}_i^{-1}\quad (i\in I_\circ),\qquad F_i\quad (i\in I_\bullet),\qquad E_i\quad (i\in I_\bullet),\qquad K_\mu\quad (\mu\in Y^\imath).
	\end{equation*}
We write $B_i=F_i$, for $i\in I_\bullet$. Here $T_{w_{\bullet}}$ denotes Lusztig's Braid group operator; cf. \cite{BW18a}*{\S2.2}.
 
	The pair $(\U, \U^\imath)$ is called a {\em quantum symmetric pair}. The subalgebra $\U^\imath$ is called an \emph{$\imath$quantum group}. In what follows, we assume the parameters $(\varsigma_i)_{i\in I_\circ}$ are taken as the same as in \cite{Wa23}*{Lemma 4.2.1}. The only reason we pick these particular parameters is that the stability conjecture of $\imath$canonical basis, or equivalently, Theorem \ref{thm:Vlabased}, holds for these parameters. 
	
	Let $\dot{\U}^\imath$ be the modified algebra of $\U^\imath$ and let $\dot{\RB}^\imath$ be the canonical basis of $\dot{\U}^\imath$ defined in \cite{BW18a}*{\S3.7 \& \S6.4}. 
Let $_\A\dot{\U}^\imath$ be the $\A$-form of $\dot{\U}^\imath$ \cite{BW18a}*{Definition~3.19}. For any $\A$-algebra $R$, set ${}_R\dot{\U}^\imath=R\otimes_{\A} {_\A\dot{\U}^\imath}$. A unital $\Uidot$-module is equivalent to a $\Ui$-module with a direct sum decomposition intro weight spaces by \cite{BW18a}*{\S3.7}. These are the only modules we shall consider in this paper.

\subsubsection{Based $\Ui$-modules} \label{sec:basedUi}
Let $\psi_\imath$ be the anti-linear bar involution on $\U^\imath$ as in \cite{BW18a}*{Lemma 3.15}. Let $(M, B)$ be a finite-dimensional based $\U$-module. We consider $M$ naturally as a $\Ui$-module by restriction. Then by \cite{BW21}*{Theorem~6.12}, $M$ admits an anti-linear involution $\psi_\imath$ such that 
\[
\psi_\imath(um) = \psi_\imath(u) \psi_\imath(m), \quad u \in \Ui, m \in M.
\]

Moreover, the $\Ui$-module $M$ admits a unique basis, called the \emph{$\imath$canonical basis}, denoted by $B^\imath = \{b^\imath \vert b \in B\}$, such that 
\begin{enumerate}   
\item $\psi_\imath (b^\imath ) = b^\imath$ for any $b^\imath$;
\item $b^\imath = b + \sum_{b' \in B} t_{b;b'}b'$, for $t_{b;b'} \in q^{-1} \BZ[q^{-1}]$;
\item $B^\imath$ forms an $\CA$-basis of the $\CA$-lattice ${}_\CA M$, and forms an $\mathbf{A}$-basis of the $\mathbf{A}$-lattice $L(M)$;
\item $(M,B^\imath)$ is a based $\Ui$-module in the sense of \cite{BW21}*{Definition~6.11}.
\end{enumerate}

In particular, we have the based $\Ui$-modules $(V(\lambda), B(\lambda)^\imath)$ for each $\lambda \in X^+$. For any based $\U$- (resp., $\U^\imath$-) module $M$, we write $_\A M$ to be the $\A$-submodule of $M$ spanned by the canonical basis (resp., $\imath$canonical basis).

For any $\lambda \in X^+$, the based $\U$-submodule  $M[\ge \lambda] \subset M$ is also a based $\Ui$-submodule by \cite{BW21}*{Theorem~6.12} (with respect to different bases). Let $\phi:M \longrightarrow \bigoplus_{i =1}^n \bigoplus_{b \in B[\lambda_i]^{hi}} V(\lambda_i)$ be the isomorphism in Lemma~\ref{lem:Ka91}. We conclude that 

{\it (a) the based $\U$-module $(M,B)$ admits a based filtration of based $\Ui$-modules
\begin{equation*} 
0 = M[\ge \lambda_0] \subset M[\ge \lambda_1] \subset M[\ge \lambda_2] \subset \cdots \subset M[\ge \lambda_n] = M, \quad \text{for }\lambda_i \in X^+,
\end{equation*}
such that 
$\phi_{\lambda_i}:M[\ge \lambda_i]/M[\ge \lambda_{i-1}] \overset{\sim}{\rightarrow} \bigoplus_{b \in B[\lambda_i]^{hi}} V(\lambda_i)$ as based $\Ui$-modules for $1 \le i \le n$; 
}

{\it (b)
for any $b \in B[\lambda_j]$, we have $\phi(b^\imath) = \phi_{\lambda_j}(b^\imath) + q^{-1}\phi(L(M[\ge \lambda_{j-1}]))$.
}

\subsubsection{Spherical modules}\label{sec:spherical}

 The following theorem was conjectured in \cite{BW18a}*{Remark~6.18} and proved in \cite{Wa23}*{Theorem~4.3.1}. 
 
We define $X^+_\imath = \{\lambda \in X^+ \vert \overline{\lambda} = \overline{0} \in X_\imath\}$.

\begin{thm} \label{thm:Vlabased}
    Let $\lambda \in X^+_\imath$. Then there exists a unique morphism as based $\Ui$-modules
    \[
     V(\lambda) \rightarrow V(0), \qquad v^+_\lambda \mapsto v^+_0.
    \]
\end{thm}

Note that $V(\lambda)$ is generated by $v_\lambda^+$ as a $\Ui$-module.
Hence for any $\lambda \in X^+$, there exists a non-zero map $V(\lambda) \rightarrow V(0)$ only if $\lambda \in X^+_\imath$. We draw the following conclusions.

{\it (a) For any $\lambda \in X^+$, there exists a unique morphism $V(\lambda) \rightarrow V(0)$ as based $\Ui$-modules
      if and only if  $\lambda \in X^+_\imath$.
}

{\it (b) For any $\lambda \in X^+$, we have 
      \[
      \dim \text{\rm Hom}_{\Ui\text{\rm -mod}}(V(\lambda), V(0)) = \begin{cases}1, &\text{if } \lambda \in X^+_\imath;\\
      0, &\text{otherwise}.
      \end{cases}  
      \]
     }
 
\subsubsection{Symmetric subgroups}\label{sec:sym}

We recall the construction of symmetric subgroups via $\imath$quantum groups  following \cite{BS22}*{\S3}.

Let $k$ be an algebraically closed field with characteristic not 2. Let $G_k$ be a reductive linear algebraic  group over $k$, and $\theta_k$ be an algebraic group involution on $G_k$. The pair $(G_k,\theta_k)$ is then called a \emph{symmetric pair}. The closed subgroup $K_k=G_k^{\theta_k}$ consisting of $\theta_k$-fixed elements in $G_k$ is called a \emph{symmetric subgroup}. By \cite{Spr87}, the classification of symmetric pairs is independent of the characteristic of $k$ (provided not $2$). By \cite{BS22}*{\S3}, one can associate an $\imath$root datum to such a symmetric pair independent of the characteristic of $k$ (provided not $2$).

On the other hand, starting with an $\imath$root datum. Let $\U$, $\U^\imath$, $\dot{\U}$, $\dot{\U}^\imath$, etc., be the various quantum algebras associated with the given datum following \S\ref{sec:qp} and \S\ref{sec:iqp}. Let $G_k$ be the reductive group associated with the underlying root datum defined over $k$. Set $$_k\dot{\U}=k\otimes_\A{}_\A\dot{\U}\qquad\text{ and }\qquad _k\dot{\U}^\imath=k\otimes_\A{}_\A\dot{\U}^\imath.$$ 
Let $_k\widehat{\U}$ be the space consisting of formal sums of $k$-linear combinations of canonical basis elements of $_k\dot{\U}$ defined in \cite{Lu07}*{\S1.1}. By \cite{BS22}*{Lemma 3.1}, the space $_k\dot{\U}^\imath$ is naturally a subspace of $_k\widehat{\U}$.
By \cite{Lu07}, one can identify elements in the coordinate ring $k[G_k]$ as $k$-linear forms on of a completion of the space $_k\dot{\U}$, that is,
    $k[G_k]\hookrightarrow \text{Hom}_k({}_k\widehat{\U},k)$.
We define 
$\mathcal{I}=\{f\in k[G_k]\mid f({}_k\dot{\U}^\imath)=0\}$.

\begin{thm}[\cite{BS22}*{Theorem 3.8}]\label{thm:ik}
    The subspace $\mathcal{I}$ is a reduced ideal of $k[G_k]$, defining a symmetric subgroup $K_k=G_k^{\theta_k}$. The $\imath$root datum associated with the pair $(G_k,\theta_k)$ is isomorphic to the one we start with.
\end{thm}

\begin{rem}
We previously established Theorem~\ref{thm:ik} for quasi-split symmetric pairs in \cite{BS22}, since Theorem~\ref{thm:Vlabased} was only available for the quasi-split cases back then. Theorem~\ref{thm:Vlabased} was later  established by Watanabe \cite{Wa23} for all symmetric pairs (of finite type). Hence results in \cite{BS22}*{\S3} hold in this generality as well; cf. \cite{BS22}*{Remark~3.3}.
\end{rem}

\section{Bases for coinvariants}

\subsection{Coinvariant spaces}\label{sec:ast}
Recall the counit 
$\varepsilon$ of $\U$ in \S\ref{sec:def}. Set $\agUi=\U^\imath\cap\text{ker }\varepsilon$ to be the augmentation ideal of $\U^\imath$. Let $(M,B)$ be a finite-dimensional based $\U$-module. We then obtain the based $\Ui$-module $(M, B^\imath)$. Let $M_+=\U^{\imath,+}M$ be the $\U^\imath$-submodule of $M$. We define 
    \[
    B^\imath_\ast  = \{b^\imath \in B^\imath \vert b \in B[\lambda]^{hi} \text{ for some } \lambda \in X_\imath^+\} \subset B^\imath.
    \]

    \begin{thm}\label{thm:Fbased}
        The $\Ui$-submodule $M_+=\agUi M$ is a based $\Ui$-submodule of $M$ with the basis $\{b^\imath \in B^\imath \vert b^\imath \not \in B^\imath_\ast\}$.
        
   In particular, the space of $\Ui$-coinvariants $M^\imath_\ast = M / \agUi M$ is naturally a based $\Ui$-module, with the basis given by the image of elements in $B_*^\imath$. The natural projection $M \rightarrow M^\imath_\ast$ maps ${}_\CA M$ to ${}_\CA M^\imath_\ast$.
    \end{thm}
   We abuse notations and denote the based quotient by $(M^\imath_\ast, B_*^\imath)$, and write $({}_\CA M)^\imath_\ast = {}_\CA M^\imath_\ast$. We also write ${}_R M^\imath_\ast = R \otimes_{\A} {}_\CA M^\imath_\ast$ for any commutative $\A$-algebra $R$.
   
    \begin{proof}
    Let $M[\ge \lambda_1] \subset M[\ge \lambda_2] \subset \cdots \subset M[\ge \lambda_n]$ be the based filtration of $M$ as \S\ref{sec:basedUi} (a).  We define $F(M) = \oplus_{b^\imath \in B^\imath_\ast} \Qq e_{b^\imath}$ to be the trivial $\Ui$-module of dimension $|B^\imath_\ast|$. By definition, we have $u \cdot x = \varepsilon(u)x$ for any $u \in \Ui$ and $x \in F(M)$. It is a based $\U^\imath$-module with basis $\{e_{b^\imath}\mid b^\imath\in B^\imath_*\}$.

    {\it (a) There is a unique $\Ui$-module homomorphism $f:M \rightarrow F(M)$ such that 
    \[
        f: b^\imath \mapsto 
    \begin{cases}
            e_{b^\imath}, &\text{if } b \in B[\lambda_i]^{hi} \text{ with } \lambda_i \in X^+_\imath;\\
            0, &\text{if } b  \in B[\lambda_i]^{hi} \text{ with } \lambda_i \notin X^+_\imath.
    \end{cases}
    \]
    }

By \cite{BW18a}*{Lemma~6.2} and an straightforward induction on $n$, we see that $M$ is generated as an $\Ui$-module by $b^\imath$ for various $b \in B[\lambda_i]^{hi}$ with $1 \le i \le n$. This shows the uniqueness. It also follows that the integral form ${}_\mathcal{A} M$ is  generated by $b^\imath$ for various $b \in B[\lambda_i]^{hi}$ with $1 \le i \le n$ as an ${}_\mathcal{A}\Uidot$-module.

We next show the existence. We firstly define $\Ui$-module homomorphisms $f_i: M \rightarrow F(M)$ for $1 \le i \le n$ as follows: if $\overline{\lambda_i} \neq \overline{0} \in X_\imath$, we define $f_i$ as the zero map; if $\overline{\lambda_i} = \overline{0} \in X_\imath$, we define $f_i$ as the composition
\begin{equation}\label{eq:fi}
    f_i:  M \xrightarrow{\phi} \bigoplus_{j =1}^n \bigoplus_{b \in B[\lambda_j]^{hi}} V(\lambda_j) \xrightarrow{p} \bigoplus_{b \in  B[\lambda_i]^{hi}} V(\lambda_i) \xrightarrow{\pi} \bigoplus_{b \in  B[\lambda_i]^{hi}} \Qq e_{b^\imath} \subset F(M),
\end{equation}
where $\phi$ is the isomorphism from Lemma~\ref{lem:Ka91}, $p$ is the canonical projection, and $\pi$ is direct sum of the based $\U$-module morphisms defined in Theorem~\ref{thm:Vlabased}.

    Let $1 \le i,j \le n$ with $\overline{\lambda_i} = \overline{0} \in X_\imath$ and $b \in B[\lambda_j]$. By Lemma~\ref{lem:Ka91}, we have 
\begin{equation}\label{eq:heart:thm:Fbased}
 \tag{$\heartsuit$} f_i(b^\imath  ) = 
 \begin{cases}
    0, &\text{if } j < i;\\
    e_{b^\imath}, &\text{if } j = i \text{ and } b \in B[\lambda_j]^{hi};\\
    0, &\text{if } j = i \text{ and } b \not \in B[\lambda_j]^{hi};\\
    c_{b^\imath}e_{b^\imath}, &\text{if }i <j, \text{ for some }  c_{b^\imath} \in q^{-1} \BZ[q^{-1}].
 \end{cases}
 \end{equation}
 
Hence a suitable $\BZ[q^{-1}]$-linear combination of $f_i$'s will be the desired map $f$. 

{\it (b) The map $f$ commutes with the bar involution $\psi_\imath$, that is, $f (\psi_\imath(x)) = \psi_{\imath}(f(x))$ for any $x \in M$.}

Since $\{b^\imath  \vert  b \in B[\lambda_i]^{hi} \}$ is a set of  generators of $M$ as a $\Ui$-module, the claim follows immediately from the fact that they are $\psi_\imath$-invariant. 

{\it (c) The map $f$ preserves the $\mathbf{A}$-lattice, that is, $f(L(M)) = L(f(M))$.}

Since all three maps $\phi$, $p$, $\pi$ in \eqref{eq:fi} preserve the $\mathbf{A}$-lattices, the map $f_i$ preverses the $\mathbf{A}$-lattices.  Since $f$ is a $\BZ[q^{-1}]$-linear combination of $f_i$'s, the claim follows.

{\it (d) The map $f$ is a morphism of based $\Ui$-modules.} 

    Thanks to $(b)$ and $(c)$, it suffices to consider the induced map $f_{\infty}: L(M)/q^{-1}L(M) \rightarrow L(F(M))/ q^{-1} L(F(M))$ at $q =\infty$; cf. \cite{Lu94}*{\S27.1.5}. 

     Let $1 \le i,j \le n$ with $\overline{\lambda_i} = \overline{0} \in X_\imath$ and $b \in B[\lambda_j]$. At $q =\infty$, we obtain from \eqref{eq:heart:thm:Fbased} that
\[
 f_{i, \infty}(b^\imath  ) = 
 \begin{cases}
    0, &\text{if } j \neq i;\\
    e_{b^\imath}, &\text{if } j = i \text{ and } b \in B[\lambda_j]^{hi};\\
    0, &\text{if } j = i \text{ and } b \not \in B[\lambda_j]^{hi}.\\
 \end{cases}
\]
Since $f$ is a $\mathbb{Z}[q^{-1}]$-linear combination of $f_i$'s, claim (d) follows.

{\it (e) We have $\ker f = M_+$.}
    
We have $M_+ \subset \ker  f$, since the action of $\Ui$ on $F(M)$ is trivial, i.e., via the counit $\varepsilon$. Then it follows by dimension counting that $M_+ = \ker  f$.

Now the theorem is proved.
    \end{proof}

Retain the same notations from Theorem~\ref{thm:Fbased}. Let $_\A M_+$ be the $\A$-submodule of $M$  spanned by elements in $B^\imath-B^\imath_*$. Then $_\A M_+ = M_+ \cap {}_\A M$ is an $_\A\dot{\U}^\imath$-module. For any commutative $\A$-algebra $R$, set $_R M_+=R\otimes_{\A}{}_\A M_+$.

\subsection{Coinvariant functors} \label{sec:F}

For any based morphism $f: (M,B)\rightarrow (M',B')$ of finite-dimensional based $\U$-modules, we naturally have the based morphism $f: (M,B^\imath)\rightarrow (M',(B')^\imath)$ of  based $\Ui$-module. We then have the induced morphism $f^\imath_\ast: M^\imath_\ast \rightarrow (M')^\imath_\ast$ of $\Ui$-modules. 

\begin{prop}\label{prop:fc}
    \begin{enumerate}
    \item The morphism $f^\imath_\ast: M^\imath_\ast \rightarrow (M')^\imath_\ast$ is a based morphism of $\Ui$-modules. 
    Therefore $(\cdot )^\imath_\ast$ is a functor from the category of finite-dimensional based $\U$-modules to the category of finite-dimensional based $\Ui$-modules.  
    \item The functor $(\cdot )^\imath_\ast$ is an exact functor.
    \end{enumerate}
\end{prop}

\begin{proof}
We show part (1). Let $\lambda \in X^+$ and $b \in B[\lambda]^{hi}$. 
Since $f$ is a morphism of based $\U$-modules, we have $f(b) \in (B'[\lambda])^{hi} \sqcup \{0\}$ by investigating the induced map at $q = \infty$. Therefore, $f$ maps $B^\imath_\ast$ to $(B')^\imath_\ast \sqcup\{0\}$. The claim now follows from Theorem~\ref{thm:Fbased}.

    We show part (2). The right exactness is trivial. To show the left exactness, suppose we have an injective based $\U$-module homomorphism $f:(M,B)\rightarrow (M',B')$. Then $f$ is automatically an injective based $\Ui$-module homomorphism $f:(M,B^\imath)\rightarrow (M',(B')^\imath)$.    
    Hence $f(b)\neq 0$ for $b\in B^\imath$. In particular $f(b)\neq 0$ for $b\in B_*^\imath$. Since $f$ maps $B^{hi}[\lambda]$ to $B'[\lambda]^{hi}$ for any $\lambda \in X^+$, we see that $f$ maps  $B_*^\imath$ to $(B')_*^\imath$. We conclude that $f^\imath_\ast$ is injective, since it is injective on the set of basis. 
\end{proof}

Recall \S\ref{sec:modiU} that for $\lambda\in X^+$, we have the short exact sequence of based $\U$-modules  \[
0\longrightarrow\dot{\U}[\not<\lambda]/\dot{\U}[\not\le\lambda]\longrightarrow\dot{\U}/\dot{\U}[\not\le\lambda]\longrightarrow \dot{\U}/\dot{\U}[\not <\lambda]\longrightarrow 0.
\]
For any $\lambda \le \lambda '$ in $X^+$, we also have the short exact sequence of based $\U$-modules  \[
0\longrightarrow\dot{\U}[\not \le \lambda]/\dot{\U}[\not\le\lambda']\longrightarrow\dot{\U}/\dot{\U}[\not\le\lambda']\longrightarrow \dot{\U}/\dot{\U}[\not \le \lambda]\longrightarrow 0.
\]

\begin{cor}\label{cor:exact}
\begin{enumerate}
\item For any $\lambda \in X^+$, we have the following short exact sequence of based $\U^\imath$-modules 
\begin{equation*}
0\longrightarrow  (\dot{\U}[\not<\lambda]/\dot{\U}[\not\le\lambda])^\imath_\ast \longrightarrow  (\dot{\U}/\dot{\U}[\not\le\lambda])^\imath_\ast \longrightarrow  (\dot{\U}/\dot{\U}[\not <\lambda])^\imath_\ast \longrightarrow 0.
\end{equation*}
\item For any $\lambda \le \lambda '$ in $X^+$, we have the short exact sequence of based $\Ui$-modules  \[
0\longrightarrow 
(\dot{\U}[\not \le \lambda]/\dot{\U}[\not\le\lambda'])^\imath_\ast
\longrightarrow
(\dot{\U}/\dot{\U}[\not\le\lambda'])^\imath_\ast
\longrightarrow (\dot{\U}/\dot{\U}[\not \le \lambda])^\imath_\ast
\longrightarrow 0.
\]
\item For any $\lambda \in X^+$, we have $(\dot{\U}/\dot{\U}[\not\le\lambda])^\imath_\ast \neq(\dot{\U}/\dot{\U}[\not <\lambda])^\imath_\ast$ if and only if $\lambda \in X^+_\imath$.
\end{enumerate} 
\end{cor}

\begin{proof}
The short exact sequences in part (1) and (2) follow from the exactness of the functor $(\cdot )^\imath_\ast$ in Proposition~\ref{prop:fc}.

If follows from \S\ref{sec:basedUi} that $\dot{\U}[\not<\lambda]/\dot{\U}[\not\le\lambda] \cong \bigoplus_{b \in \RB[\lambda]^{hi}} V(\lambda)$. Then by \S\ref{sec:spherical}, we see  $(\dot{\U}[\not<\lambda]/\dot{\U}[\not\le\lambda])^\imath_\ast  \neq 0$ if and only if $\lambda \in X^+_\imath$. Then part (3) follows from the short exact sequence in (1) now.
\end{proof}

    
\section{Coordinate rings of symmetric spaces}
\subsection{The quantum coordinate ring $\mathbf{O}(K \backslash G)$}

 Let $\mathbf{O}(K \backslash G)$ be the $\Qq$-subspace of $\mathbf{O}$ consisting of linear forms on $\Udot$ that vanish on the subspace $\agUi \Udot$. 

\begin{prop}\label{prop:pc}
    \begin{enumerate} 
    \item The subspace $\mathbf{O}(K \backslash G)$ is closed under multiplication. 
    \item The coproduct $\Delta$ of $\mathbf{O}$ restricts to $\Delta:\mathbf{O}(K \backslash G)\rightarrow \mathbf{O}(K \backslash G) \otimes \mathbf{O}$. 
    \end{enumerate}
 
\end{prop}

\begin{proof}
Part (1) follows if we can show 
$
    \Delta(\agUi) \subset \agUi \otimes \U+ \U \otimes \agUi.
$
Since the subspace $\agUi$ is the two-sided ideal of $\U^\imath$, generated by $B_i$ ($i\in I$), $E_i$ ($i\in I_\bullet$) and $K_\mu-1$ ($\mu\in Y^\imath$), it will suffice to check $\Delta(x)\in \agUi \otimes \U+ \U \otimes \agUi$, when $x$ is a generator. This is direct and will be skipped. Part (2) follows from the fact that $\agUi\dot{\U}$ is a right ideal of $\dot{\U}$
\end{proof}

Let ${}_\CA \mathbf{O}(K \backslash G) = \mathbf{O}(K \backslash G) \cap {}_\CA \mathbf{O}$. Thanks to Proposition \ref{prop:pc} (1), we see that $_\CA\mathbf{O}(K \backslash G)$ is an $\CA$-subalgebra of $_\CA\mathbf{O}$. For any commutative $\CA$-algebra $R$, we write ${}_R \mathbf{O}(K \backslash G) = R \otimes_{\CA} {}_\CA \mathbf{O}(K \backslash G)$.
For any $\lambda \in X^+$, we define $\mathbf{O}(K \backslash G)_{\le \lambda} = \mathbf{O}(K \backslash G) \cap \mathbf{O}_{\le \lambda}$ and ${}_\CA \mathbf{O}(K \backslash G)_{\le \lambda} = \mathbf{O}(K \backslash G) \cap {}_\CA\mathbf{O}_{\le \lambda}$. We similarly define $\mathbf{O}(K \backslash G)_{< \lambda}$ and ${}_\CA \mathbf{O}(K \backslash G)_{< \lambda}$.

\subsection{Canonical basis on $\mathbf{O}(K \backslash G)$}\label{sec:BKG}
For any  $\lambda \in X^+$, we have the finite-dimensional based $\U$-module $(\Udot/ \Udot[\not \le \lambda], \dot{\RB}[\le \lambda])$ by \S\ref{sec:modiU}. Then $\Udot/ \Udot[\not \le \lambda]$ is naturally a based $\Ui$-module with basis $\dot{\RB}[\le \lambda]^\imath$. 

We define $\dot{\RB}[\le \lambda]^\imath_\ast$ following \S\ref{sec:ast}. For any $a \in \dot{\RB}[\le \lambda]$, and hence $a^\imath \in \dot{\RB}[\le \lambda]^\imath$, we define the $\Qq$-linear form 
\[
\widetilde{a}^\imath : \Udot/ \Udot[\not \le \lambda] \rightarrow \Qq,  \quad \text{such that }\widetilde{a}^\imath(b^\imath) = \delta_{a^\imath, b^\imath} \text{ for all } b^\imath \in \dot{\RB}[\le \lambda]^\imath.
\]
Then we see that $\{\widetilde{a}^\imath \vert a^\imath \in \dot{\RB}[\le \lambda]^\imath\}$ forms a $\Qq$-basis of $\mathbf{O}_{\le \lambda}$ and an $\CA$-basis of ${}_\CA \mathbf{O}_{\le \lambda}$.

Let $\lambda, \lambda' \in X^+$ be such that $\lambda' \ge \lambda$. The natural projection
$\Udot/\Udot[\not \le \lambda'] \rightarrow \Udot/\Udot[\not \le \lambda]$ is a morphism of based $\U$-modules. It is then automatically a morphism of based $\Ui$-modules by \cite{BW18a}*{\S6}
\[
 (\Udot/\Udot[\not \le \lambda'], \dot{\RB}[\le \lambda']^\imath) \rightarrow (\Udot/\Udot[\not \le \lambda], \dot{\RB}[\le \lambda]^\imath ).
\]
Therefore 

{\it (a) the union $\cup_{\lambda \in X^+}\{\widetilde{a}^\imath \vert a^\imath \in \dot{\RB}[\le \lambda]^\imath\}$ forms a $\Qq$-basis of $\mathbf{O}$ and an $\CA$-basis of ${}_\CA \mathbf{O}$.
}

\begin{rem}
There is no natural $\imath$canonical basis $\dot{\RB}^\imath$ on the whole modified form $\Udot$. One could consider a topological basis via certain completion similar to \cite{Lu94}*{Chapter 30}. However, due the finiteness assumption on the elements in $\mathbf{O}$, we do naturally have the dual basis $\cup_{\lambda \in X^+}\{\widetilde{a}^\imath \vert a^\imath \in \dot{\RB}[\le \lambda]^\imath\}$. 
\end{rem}

\begin{lem}\label{lem:free1}
    The set $\RB(K\backslash G)_{\le\lambda}=\{\widetilde{a}^\imath \vert a^\imath \in \dot{\RB}[\le \lambda]^\imath_\ast\}$ is a $\Qq$-basis of $\mathbf{O}(K \backslash G)_{\le \lambda}$ and an $\CA$-basis of ${}_\CA \mathbf{O}(K \backslash G)_{\le \lambda}$. 
    In particular, ${}_\CA \mathbf{O}(K \backslash G)_{\le \lambda}$ is a free $\CA$-module. 
    
    Similar statements hold for $ \mathbf{O}(K \backslash G)_{< \lambda}$ and $_\CA \mathbf{O}(K \backslash G)_{< \lambda}$.
\end{lem}

\begin{proof}
 Recall the exact functor $(\cdot)^\imath_\ast$ in \S\ref{sec:F}. By definition, we have 
 \begin{align*}
     \mathbf{O}(K \backslash G)_{\le \lambda} 
     & = \{f \in {\rm Hom}_{\Qq}\Big( (\Udot/ \Udot[\not \le \lambda] ), \Qq\Big) \mid f \Big(  (\Udot/ \Udot[\not \le \lambda])_+ \Big)=0\} \\
     &= {\rm Hom}_{\Qq}\Big( (\Udot/ \Udot[\not \le \lambda] )^\imath_\ast, \Qq \Big).
 \end{align*}
The lemma follows immediately by Theorem~\ref{thm:Fbased}.
\end{proof}

By dualizing Corollary~\ref{cor:exact}, we have the natural embeddings ($\text{for }\lambda \le \lambda'$ in $X^+$)
\begin{align*}
 \mathbf{O}(K \backslash G)_{\le \lambda} \rightarrow \mathbf{O}(K \backslash G)_{\le \lambda'} &\text{ mapping } B(K\backslash G)_{\le\lambda} \text{ into }B(K\backslash G)_{\le\lambda'},\\
 \mathbf{O}(K\backslash G)_{<\lambda}\rightarrow\mathbf{O}(K\backslash G)_{\le\lambda} &\text{ mapping } B(K\backslash G)_{< \lambda} \text{ into }B(K\backslash G)_{\le\lambda}.
\end{align*}
We set $\RB(K\backslash G)_\lambda=\RB(K\backslash G)_{\le\lambda}-\RB(K\backslash G)_{<\lambda}$.

\begin{thm}\label{thm:base}
    The set 
    \[
   \RB(K \backslash G)=  \bigcup_{\lambda \in X^+}  \RB(K\backslash G)_{\le\lambda}=\bigsqcup_{\lambda\in X^+}\RB(K\backslash G)_\lambda = \bigsqcup_{\lambda\in X_\imath^+}\RB(K\backslash G)_\lambda
    \]
    is a $\Qq$-basis of $ \mathbf{O}(K \backslash G)$ and an $\CA$-basis of ${}_\CA \mathbf{O}(K \backslash G)$. 
    
    Moreover, for any commutative $\CA$-algebra $R$ with $1$, the ring ${}_R \mathbf{O}(K \backslash G) =R\otimes_{\CA} {}_\CA \mathbf{O}(K \backslash G)$ is naturally a free $R$-submodule of $ {}_R \mathbf{O} =R\otimes_{\CA} {}_\A\mathbf{O}$.
\end{thm}

\begin{proof}
 We have $\RB(K\backslash G)_\lambda \neq \emptyset$ if and only if $\lambda \in X^+_\imath$ by Corollary~\ref{cor:exact}.  The first claim follows immediately.

We know ${}_\CA \mathbf{O}(K \backslash G) \subset {}_\CA \mathbf{O}$ be definition. Since ${}_\CA \mathbf{O}(K \backslash G)$ is spanned by the subset 
$\RB(K \backslash G)$ of the $\CA$-basis $\cup_{\lambda \in X^+}\{\widetilde{a}^\imath \vert a^\imath \in \dot{\RB}[\le \lambda]^\imath\}$ of ${}_\CA \mathbf{O}$, the injectivity is clearly preserved after the base change $R \otimes_{\CA} -$.
\end{proof}
We call $\RB(K \backslash G)$ the canonical basis of $\mathbf{O}(K \backslash G)$. 

\begin{cor}\label{cor:coi}
    The coproduct of ${}_\A\mathbf{O}$ induces the coproduct on the $\A$-forms 
    \begin{equation}
        _\A\mathbf{O}(K\backslash G)\longrightarrow {_\A\mathbf{O}}(K\backslash G)\otimes_\A {}_\A\mathbf{O}.
    \end{equation}
\end{cor}
\begin{proof}
It is known that the coproduct restricts to integral forms $\Delta: {}_\A\mathbf{O}\rightarrow{}_\A\mathbf{O}\otimes_\A{}_\A\mathbf{O}$. The the corollory follows from Proposition \ref{prop:pc} and the fact that $_\A\mathbf{O}(K\backslash G)$ is spanned by a subset the basis of $_\A\mathbf{O}$ in \S\ref{sec:BKG} (a).
\end{proof}

\subsection{The right action}\label{sec:right}

For $\lambda \in X^+$, since $\dot{\U}[\not\leq\lambda]$ is a two-sided ideal of $\dot{\U}$, the space $\dot{\U}/\dot{\U}[\not\le\lambda]$ is a $(\U,\U)$-bimodule. One can define the notion of based right $\U$-modules in parallel with \cite{Lu94}*{\S 27}. It follows by \cite{Ka94} that  $(\dot{\U}/\dot{\U}[\not\le\lambda],\dot{\RB}[\le\lambda])$ is also a based right $\U$-module. By the construction of the functor $(\cdot)^\imath_\ast$, the space $ (\dot{\U}/\dot{\U}[\not\le\lambda])^\imath_\ast$ admits a right $\U$-action. Recall that $(\dot{\U}/\dot{\U}[\not\le\lambda])^\imath_\ast$ is a (left) based $\U^\imath$-module with $\imath$canonical basis ${\RB}[\le \lambda]^\imath_\ast$.

\begin{lem}\label{le:rb}
\begin{enumerate}
    \item $(\dot{\U}/\dot{\U}[\not\le\lambda], \dot{\RB}[\le \lambda]^\imath)$ is naturally a based right $\U$-module.
   \item  The right $\U$-modules  $((\dot{\U}/\dot{\U}[\not\leq\lambda])^\imath_\ast, {\RB}[\le \lambda]^\imath_\ast)$ and $( (\dot{\U}/\dot{\U}[\not < \lambda])^\imath_\ast,\dot{\RB}[\not <\lambda]_*^\imath)$ are based. 
    \end{enumerate}
\end{lem}

\begin{proof}

We show part (1).
    We verify (the right module variant of the) conditions in \cite{Lu94}*{\S27.1.2}. Conditions (a), (b) are trivial. Since the canonical basis and $\imath$canonical basis coincide at $q=\infty$, condition (d) follows. We define the bar involution on $\dot{\U}/\dot{\U}[\not\le\lambda]$ by $\psi_\imath=\Upsilon\psi$ following \cite{BW18a}*{Proposition~5.1}, where $\psi$ is the bar involution fixing canonical basis, and $\Upsilon$ is an element in a certain completion of $\U$. The $\imath$canonical basis $\dot{\RB}[\leq\lambda]^\imath$ is fixed by $\psi_\imath$ by definition. For $u\in \dot{\U}/\dot{\U}[\not\le \lambda]$ and $x\in \U$, we have
    \[
    \psi_\imath(ux)=\Upsilon\psi(ux)=\Upsilon\psi(u)\psi(x)=\psi_\imath(u)\psi(x).
    \]
    Hence the right $\U$-action on $\dot{\U}/\dot{\U}[\not\le\lambda]$ is bar-equivariant with respect to $\psi_\imath$.  This verifies condition (c) in \cite{Lu94}*{\S27.1.2}.

    Now it follows from Theorem \ref{thm:Fbased} that $(\dot{\U}/\dot{\U}[\not\le\lambda])_+$ is a based right $\U$-submodule, and $(\dot{\U}/\dot{\U}[\not\le\lambda])^\imath_\ast$ is a based quotient of right $\U$-submodule.
\end{proof}

The following short exact sequence of based $\U^\imath$-modules in Corollary~\ref{cor:exact}
\begin{equation*}
0\longrightarrow  (\dot{\U}[\not<\lambda]/\dot{\U}[\not\le\lambda])^\imath_\ast \longrightarrow  (\dot{\U}/\dot{\U}[\not\le\lambda])^\imath_\ast \longrightarrow  (\dot{\U}/\dot{\U}[\not <\lambda])^\imath_\ast \longrightarrow 0
\end{equation*}
is also an exact sequence as based right $\U$-modules, thanks to Lemma \ref{le:rb}.

Recall the anti-algebra automorphism $\sigma$ and the algebra homomorphism $\omega$ in \S \ref{sec:def}. For $\lambda\in X^+$, let $^{\sigma \omega} V(\lambda)$ be the right $\U$-module with the same underlying vector space as $V(\lambda)$, and with the $\U$-action twisted by ${\sigma \circ \omega}$. It was proved in \cite{Ka94} the canonical basis $\dot{\RB}$ is $\sigma$-stable; it was proved in \cite{Lu22} that the canonical basis $\dot{\RB}$ is $\omega$-stable.
Then we can conclude that $({}^{\sigma \omega} V(\lambda), B(\lambda))$ is a based right $\U$-module.
 
\begin{prop}\label{prop:quo}
    For each $\lambda\in X^+$, we have an isomorphism of based right $\U$-modules,
    \[
    \left( (\dot{\U}[\not<\lambda]/\dot{\U}[\not\le\lambda])^\imath_\ast,\dot{\RB}[\lambda]_*^\imath\right)\cong \begin{cases}
        \left({}^{\sigma \omega} V(\lambda),{} B(\lambda)\right), &\text{if }\lambda \in X^+_\imath;\\
        0, & \text{if }\lambda \notin X^+_\imath.
    \end{cases}
    \]
\end{prop}

\begin{proof}
    By \cite{Lu94}*{Proposition 29.2.2 \& Proposition~29.3.1}, we have the $(\U,\U)$-bimodule isomorphism
\[
\dot{\U}[\not<\lambda]/\dot{\U}[\not\le\lambda]\cong V(\lambda)\boxtimes {}^{\sigma \omega} V(\lambda),
\]
such that we have the induced bijection of bases $\RB[\lambda] \cong  B(\lambda) \times  B(\lambda)$. Here $\boxtimes$ denotes the external tensor product emphasizing the bimodule structure.

Note that $V(\lambda)^\imath_\ast \cong \Qq$ if $\lambda \in X^+_\imath$; and $V(\lambda)^\imath_\ast $ is the zero module if $\lambda \notin X^+_\imath$ by \S\ref{sec:spherical}. Then we have the following right $\U$-module isomorphisms
\[
\big(\dot{\U}[\not<\lambda]/\dot{\U}[\not\le\lambda]\big)^\imath_\ast \cong V(\lambda)^\imath_\ast \boxtimes {}^{\sigma \omega} V(\lambda) \cong \begin{cases}
    {}^{\sigma \omega} V(\lambda), & \text{if }\lambda \in X^+_\imath;\\
    0, & \text{if }\lambda \notin X^+_\imath.
\end{cases}
\]
The compactibility on bases is clear. We finish the proof.
\end{proof}

\subsection{Symmetric spaces}
 
Let $k$ be an algebraically closed field of  characteristic $\neq 2$.  Let $\CA \rightarrow k$ be the ring homomorphism given by $q \mapsto 1$. Set $$_k\mathbf{O}=k\otimes_\A{_\A\mathbf{O}}\qquad\text{ and }\qquad _k\mathbf{O}(K\backslash G)=k\otimes_\A{_\A\mathbf{O}(K\backslash G)}.$$
The $k$-algebra ${}_k \mathbf{O}(K \backslash G)$ is naturally a subalgebra of ${}_k \mathbf{O}$ by Theorem~\ref{thm:base}.
For $\lambda\in X^+$, set $_k\mathbf{O}(K\backslash G)_{\le\lambda}=k\otimes_\A{}_\A\mathbf{O}(K\backslash G)_{\leq\lambda}$. Still by Theorem \ref{thm:base}, the $k$-algebra $_k\mathbf{O}(K\backslash G)_{\le\lambda}$ is a $k$-subalgebra of $_k\mathbf{O}$, and moreover we have $_k\mathbf{O}(K\backslash G)_{\le\lambda}={}_k\mathbf{O}(K\backslash G)\cap {}_k\mathbf{O}_{\leq\lambda}$. As a result, we have 
\begin{equation}\label{eq:KG}
_k\mathbf{O}(K\backslash G)=\cup_{\lambda\in X^+}{}_k\mathbf{O}(K\backslash G)_{\le\lambda}.
\end{equation}

Let $(G_k,\theta_k)$ be the symmetric pair associated with our $\imath$root datum, as constructed in \S \ref{sec:sym}. Recall that by Theorem \ref{thm:chs}, the $k$-algebra $_k\mathbf{O}$ is canonically isomorphic to the algebra $k[G_k]$ of regular functions on $G_k$. Let $K_k=G_k^{\theta_k}$ be the symmetric subgroup.

\begin{thm}\label{thm:KG}
    As regular functions on $G_k$, we have 
    \begin{equation}\label{eq:ok}
    _k\mathbf{O}(K\backslash G)=\{ f\in {}_k\mathbf{O}\mid f(hg)=f(g),\text{ for any }h\in K_k,\; g\in G_k\}.
    \end{equation}
    Therefore ${}_k \mathbf{O}(K \backslash G)$ is isomorphic to the coordinate ring of the symmetric space $K_k \backslash G_k$. Moreover, the basis $B(K \backslash G)$ of $_\CA\mathbf{O}(K\backslash G)$ specializes to a basis of $_k\mathbf{O}(K\backslash G)$.
\end{thm}

\begin{proof}

For $\lambda\in X^+$, let us write
\[
^{K_k} k[ G_k]_{\leq\lambda}=\{ f\in k[G_k]_{\leq\lambda}\mid f(hg)=f(g),\text{ for any }h\in K_k,\; g\in G_k\}.
\]
It will suffice to prove that $$^{K_k}k[ G_k]_{\leq\lambda}={}_k\mathbf{O}(K\backslash G)_{\leq\lambda},\qquad \text{for any }\lambda\in X^+.$$

Let $\mathcal{I}\subset k[G_k]$ be the defining ideal of the closed subgroup $K_k$, and $\delta:k[G_k]\rightarrow k[G_k]\otimes k[G_k]$ be the coproduct of $k[G_k]$. Then we have
\[
^{K_k}k[ G_k]_{\leq\lambda}=\{f\in k[G_k]_{\leq\lambda}\mid \delta(f)-1\otimes f\in \mathcal{I}\otimes k[G_k]\}.
\]
By \cite{Lu07}*{\S 3.1}, under the isomorphism $k[G_k]\cong {}_k\mathbf{O}$, the constant function 1 is the dual canonical basis element $\widetilde{1_0}$. 
Recall from Theorem \ref{thm:ik} that 
\[
\mathcal{I}=\{ f\in k[G_k]\mid f({}_k\dot{\U}^\imath)=0\}.
\]

Therefore we have
\begin{align*}
^{K_k}k[ G_k]_{\leq\lambda}&=\{f\in k[G_k]_{\leq\lambda}\mid f(yx)=\widetilde{1_0}(y)f(x),\;\text{for }y\in {}_k\dot{\U}^\imath,\;x\in {}_k\dot{\U}\}\\
&=\{ f\in \text{Hom}_k\big( {}_k(\dot{\U}/\dot{\U}[\not\le\lambda]),k\big)\mid f\big({}_k(\dot{\U}/\dot{\U}[\not\le\lambda])_+\big)=0\}.
\end{align*}
Here the last equality follows from the fact that, the integral form of the based $\U^\imath$-module $_\A(\dot{\U}/\dot{\U}[\not\le\lambda])_+$ is spanned by $yx-\widetilde{1_0}(y)x$, for $y\in{}_\A\dot{\U}^\imath$, $x\in {}_\A(\dot{\U}/\dot{\U}[\not\le\lambda])$, as an $\A$-module.

On the other hand, recall that 
\[
_\A\mathbf{O}(K\backslash G)_{\leq\lambda}=\{ f\in \text{Hom}_\A\big( {}_\A(\dot{\U}/\dot{\U}[\not\le\lambda]),\A\big)\mid f\big({}_\A(\dot{\U}/\dot{\U}[\not\le\lambda])_+\big)=0\}.
\]
By Theorem \ref{thm:Fbased}, we have
\[
_k\mathbf{O}(K\backslash G)_{\leq\lambda}=\{ f\in \text{Hom}_k\big( {}_k(\dot{\U}/\dot{\U}[\not\le\lambda]),k\big)\mid f\big({}_k(\dot{\U}/\dot{\U}[\not\le\lambda])_+\big)=0\}.
\]
This completes the proof of the equality \eqref{eq:ok}.

For the last assertion, it is known that the symmetric space $K_k\backslash G_k$ is affine by \cite{Sp83}*{Proposition 2.2}. Therefore the coordinate ring $k[K_k\backslash G_k]$ is isomorphic to the subalgebra of $K_k$-invariant functions in $k[G_k]$ (cf. \cite{Spr09}*{Exercise 5.5.9 (8)}). We complete the proof of the theorem.
\end{proof}

Thanks to Theorem \ref{thm:KG}, the base change of the coproduct in Corollary \ref{cor:coi} gives the map 
\[
k[K_k\backslash G_k]\longrightarrow k[K_k\backslash G_k]\otimes_k k[G_k],
\]
which is the comorphism of the action map
$K_k\backslash G_k\times_k G_k\longrightarrow K_k\backslash G_k$.

We hence have the following geometric reformulation of Theorem~\ref{thm:KG}. 

{\it (a) The space $Spec\, {}_\BZ\mathbf{O}(K \backslash G)$ admits a natural right action of the Chevalley group scheme $G_{\BZ}$, whose geometric fiber over any algebraically closed field $k$ of characteristic not $2$ is isomorphic to the symmetric space $K_k \backslash G_k$ with the natural right action of $G_k$.
}

\subsection{Filtrations}

The coordinate ring $_k\mathbf{O}(K\backslash G)\cong k[K_k\backslash G_k]$ admits a left $G_k$-action by 
\[
    (g \cdot f)(x) = f(xg), \quad \text{ for } x \in K_k \backslash G_k, f \in k[K_k\backslash G_k], g \in G_k.
\] 
Recall the Peter--Weyl filtration $\{{}_k\mathbf{O}_{\leq\lambda} \mid\lambda\in X^+\}$ of the coordinate ring $_k\mathbf{O}\cong k[G_k]$ in Theorem~\ref{thm:PW}. For $\lambda \in X^+$, recall 
\[
{}_k\mathbf{O}(K\backslash G)_{\leq \lambda}={}_k\mathbf{O}(K\backslash G)\cap {}_k\mathbf{O}_{\leq \lambda}.
\]
Since each subspace ${}_k\mathbf{O}_{\leq \lambda}$ is a ($G_k, G_k$)-subbimodule of $_k\mathbf{O}$, we deduce that ${}_k\mathbf{O}(K\backslash G)_{\leq \lambda}$ is a $G_k$-submodule. Similarly we have the $G_k$-submodule ${}_k\mathbf{O}(K\backslash G)_{<\lambda}$.

\begin{thm}\label{thm:fil}
    For $\lambda\in X^+$. We have the $G_k$-module isomorphisms 
    \[
    _k\mathbf{O}(K\backslash G)_{\leq \lambda}/{}_k\mathbf{O}(K\backslash G)_{<\lambda}
    \cong
    \begin{cases}
    _kV(-w_0\lambda)^*, &\text{if }\lambda \in X^+_\imath;\\
    0, &\text{if }\lambda \notin X^+_\imath.
    \end{cases}
    \]
\end{thm}

\begin{proof}
    By the construction of canonical basis of $_\A\mathbf{O}(K\backslash G)$, we have the following canonical isomorphism as vector spaces,
    \begin{align*}
    _k\mathbf{O}(K\backslash G)_{\leq \lambda}/{}_k\mathbf{O}(K\backslash G)_{< \lambda}
    &\cong k\otimes_\A\big({}_\A\mathbf{O}(K\backslash G)_{\leq\lambda}/{}_\A\mathbf{O}(K\backslash G)_{<\lambda}\big)\\
    &\cong k\otimes _\A \left(\text{Hom}_\A\big({}_\A (\dot{\U}[\not <\lambda]/\dot{\U}[\not\leq\lambda])^\imath_\ast,\A\big)\right)\\
    &\cong \text{Hom}
_k\big({}_k  (\dot{\U}[\not <\lambda]/\dot{\U}[\not\le\lambda])^\imath_\ast,k\big).
\end{align*}
The second identity follows from the exact sequence in Corollary~\ref{cor:exact}. The left $G_k$-action is given by 
\[
 (g\cdot f) (x) = f(xg), \quad g \in G_K, x \in \big({}_k  (\dot{\U}[\not <\lambda]/\dot{\U}[\not\le\lambda])^\imath_\ast, f \in \text{Hom}
_k\big({}_k  (\dot{\U}[\not <\lambda]/\dot{\U}[\not\le\lambda])^\imath_\ast,k\big).
\]

By Proposition \ref{prop:quo}, we have the right $G_k$-module isomorphism 
\[
_k  (\dot{\U}[\not <\lambda]/\dot{\U}[\not\le\lambda])^\imath_\ast \cong 
\begin{cases}
{}^{\sigma \omega}_k V(\lambda), &\text{if }\lambda \in X^+_\imath;\\
0, &\text{if }\lambda \notin X^+_\imath.
\end{cases}
\]

 Note that $\sigma\circ \omega (g)=\text{Ad}_h\circ \omega (g^{-1})$ for any $g\in G_k$ and some $h$ in a maximal torus of $G_k$. 
Therefore as left $G_k$-modules, we have
\[
_k\mathbf{O}(K\backslash G)_{\leq \lambda}/{}_k\mathbf{O}(K\backslash G)_{< \lambda} \cong 
\begin{cases}
{}_k V(-w_0\lambda)^*, &\text{if }\lambda \in X^+_\imath;\\
0, &\text{if }\lambda \notin X^+_\imath.
\end{cases} \qedhere
\]
\end{proof}

For a rational $G_k$-module $V$, an ascending chain $F^0=(0)\subset F^1\subset F^2\subset\cdots $ of $G_k$-submodules
is called a \emph{good filtration} if $\cup_j F^j=V$, and for each $j\geq 1$, the quotient $F^j/F^{j-1}$ is isomorphic to $_kV(\lambda_j)^*$  for some $\lambda_j\in X^+$. 

By Theorem \ref{thm:PW}, $k[G_k]$ has a good filtration as a $G_k\times G_k$-module. Theorem~\ref{thm:fil} generalises this result to arbitrary symmetric pairs, which we restate as follows.

\begin{cor}
    The coordinate ring of the symmetric space $K_k\backslash G_k$ has a good filtration as a $G_k$-module.
\end{cor}
 
\begin{remark}
Let $B_k \subset G_k$ be a Borel subgroup with the unipotent radical $N_k \subset B_k$. Then the similar arguments in the proof of Theorem~\ref{thm:KG} and Theorem~\ref{thm:fil} can be applied to give a construction of an integral form and a good filtration of the coordinate ring $k[N_k \backslash G_k]$.

The construction of the (integral form of) ring $k[N_k \backslash G_k]$ will actually be easier, since one only need to consider Lusztig's (dual) canonical bases.  The good filtration of the $G_k$-module $k[N_k \backslash G_k]$ actually splits, thanks to the left multiplication by the maximal torus.  
\end{remark}

\subsection{Spherical functions}
Let $\mathbf{O}(K\backslash G /K)$ be the $\Qq$-subspace of $\mathbf{O}$ consisting of linear forms on $\Udot$ that vanish on the subspace $\agUi \Udot + \Udot \agUi$. The following proposition can be proved similarly to  Proposition~\ref{prop:pc}, whose proof we omit.

\begin{prop}
    The subspace $\mathbf{O}(K\backslash G /K)$ is closed under multiplication.
\end{prop}

We define ${}_\mathcal{A}\mathbf{O}(K\backslash G /K) = \mathbf{O}(K\backslash G /K) \cap {}_\mathcal{A} \mathbf{O}$, and $\mathbf{O}(K\backslash G /K)_{\le \lambda} = \mathbf{O}(K\backslash G /K) \cap \mathbf{O}_{\le \lambda}$ for $\lambda\in X^+$. We similarly define $\mathbf{O}(K\backslash G /K)_{< \lambda}$, ${}_\mathcal{A}\mathbf{O}(K\backslash G /K)_{\le \lambda}$, and ${}_\mathcal{A}\mathbf{O}(K\backslash G /K)_{< \lambda}$.

By Lemma~\ref{le:rb} (1), for $\lambda\in X^+$, $(\Udot/ \Udot[\not \le \lambda ], \RB[\le \lambda]^\imath)$ is a based right $\U$-module, hence a based right $\U^\imath$-module. We denote its $\imath$canonical basis by ${}^\imath \RB[\le \lambda]^\imath = \{{}^\imath b ^\imath \vert b \in \RB[\le \lambda]\}$ as a right $\U^\imath$-module. Following \S\ref{sec:BKG}, we define the $\Qq$-linear form 
\[
{}^\imath\widetilde{a}^\imath : \Udot/ \Udot[\not \le \lambda] \rightarrow \Qq,  \quad {}^\imath\widetilde{a}^\imath({}^\imath b^\imath) = \delta_{{}^\imath a^\imath, {}^\imath b^\imath} \text{ for all } {}^\imath b^\imath \in {}^\imath\dot{\RB}[\le \lambda]^\imath.
\]
Then $\{{}^\imath\widetilde{a}^\imath \vert {}^\imath a^\imath \in {}^\imath\dot{\RB}[\le \lambda]^\imath \}$ forms a $\Qq$-basis of $\mathbf{O}_{\le \lambda}$ and an $\CA$-basis of ${}_\CA \mathbf{O}_{\le \lambda}$.

By Lemma~\ref{le:rb} (2), $\big((\Udot/ \Udot[\not \le \lambda ])^\imath_\ast, \RB[\le \lambda]^\imath_\ast \big)$ is a based right $\U$-module. We denote its $\imath$canonical basis by ${}^\imath \RB[\le \lambda]^\imath_\ast$ as a right $\U^\imath$-module. By the right variant of Theorem~\ref{thm:Fbased}, we see that $(\Udot/ \Udot[\not \le \lambda ])^\imath_\ast \agUi$ is a based right $\Ui$-submodule of $(\Udot/ \Udot[\not \le \lambda ])^\imath_\ast$. By applying the right variant of the construction in \S \ref{sec:ast}, we get a subset $^\imath_\ast \RB[\le \lambda]^\imath_\ast$ of the basis $^\imath \RB[\le\lambda]^\imath$, such that the based quotient $\big({}^\imath_\ast(\Udot/ \Udot[\not \le \lambda ])^\imath_\ast, {}^\imath_\ast \RB[\le \lambda]^\imath_\ast \big)$ is a based right $\U^\imath$-module. Similar to \S\ref{sec:BKG}, we see that $\RB(K \backslash G / K)_{\le \lambda} = \{{}^\imath\tilde{a}^{\imath}\vert {}^\imath\tilde{a}^{\imath} \in {}^\imath_\ast \RB[\le \lambda]^\imath_\ast\}$ is a $\Qq$-basis of $\mathbf{O}(K\backslash G /K)_{\le \lambda} $ and an $\mathcal{A}$-basis of ${}_\mathcal{A}\mathbf{O}(K\backslash G /K)_{\le \lambda}$. For any $\lambda \in X^+$, the natural embedding
\[
 \mathbf{O}(K \backslash G / K)_{< \lambda} \rightarrow \mathbf{O}(K \backslash G / K)_{\le \lambda} \text{ maps } \RB(K \backslash G / K)_{< \lambda} \text{ to } \RB(K \backslash G / K)_{\le \lambda}.
\]
We define $$\RB(K \backslash G / K)_\lambda =\RB(K \backslash G / K)_{\le \lambda} - \RB(K \backslash G / K)_{< \lambda} \quad  \text{ and } \quad \RB(K \backslash G / K) = \bigsqcup_{\lambda \in X^+} \RB(K \backslash G / K)_\lambda.$$ The following theorem is analogous to Theorem~\ref{thm:base} and Theorem~\ref{thm:KG}. The proof will be omitted.
\begin{thm}
    \begin{enumerate}
    \item The set $\RB(K\backslash G /K)$ is a $\Qq$-basis of $\mathbf{O}(K\backslash G /K) $ and an $\mathcal{A}$-basis of ${}_\mathcal{A}\mathbf{O}(K\backslash G /K)$.
    \item For any algebraically closed field $k$ of char $\neq 2$, we have 
    \[
    {}_{k} \mathbf{O}(K\backslash G /K) = \{f \in  {}_{k} \mathbf{O} \vert f(hg)=f(gh)=f(g) , \forall g \in G_k, h\in K_k\}.
    \] 
    Here $ {}_{k} \mathbf{O}(K\backslash G /K) = k \otimes_\CA  {}_\mathcal{A}\mathbf{O}(K\backslash G /K)$.
\end{enumerate}
\end{thm}

We call $\RB(K\backslash G /K)$ the canonical basis of $\mathbf{O}(K\backslash G /K) $.

\end{document}